\documentclass[11pt,a4paper]{amsart} 
\usepackage{amsmath,amsfonts,latexsym,amsthm,amssymb}
\usepackage{graphicx}
\usepackage{color}
\usepackage[text={15cm,24cm},top=3cm]{geometry}  

\newcommand{\N}{{\mathbb{N}}}  
\newcommand{\R}{{\mathbb{R}}}  
\newcommand{\id}{{\mathrm{d}}}  
\newcommand{\cla}{{\langle \lambda \rangle}} 

\theoremstyle{plain}
\newtheorem{thm}{Theorem}[section]
\newtheorem{lem}[thm]{Lemma}
\newtheorem{prop}[thm]{Proposition}

\theoremstyle{remark}

\theoremstyle{definition}

\title[Riesz means of the counting function]{Riesz means of the counting function of the Laplace operator on compact manifolds 
of non-positive curvature}

\author[K. Mroz]{Kamil Mroz}
\address{Department of Mathematical Sciences,  Loughborough University,  Loughborough, Leicestershire, LE11 3TU,
UK}
\email{k.mroz@lboro.ac.uk}

\author[A. Strohmaier]{Alexander Strohmaier}
\address{Department of Mathematical Sciences,  Loughborough University,  Loughborough, Leicestershire, LE11 3TU,
UK} \email{a.strohmaier@lboro.ac.uk}

\begin{document}

\maketitle

\begin{abstract}
Let $(M, {g})$ be a compact, $d$-dimensional Riemannian manifold without boundary. 
Suppose further that $(M,g)$ is either two dimensional and has 
no conjugate points or $(M,g)$ has non-positive sectional curvature. The goal of this 
note is to show that the long time parametrix obtained for such manifolds by B\'erard 
can be used to prove a logarithmic improvement for the remainder term of the Riesz 
means of the counting function of the Laplace operator.
\end{abstract}
\vspace{1cm}
Let $(M,g)$ be a closed Riemannian manifold. Denote by 
$\Delta$ the (geometric) Laplace operator on functions, given in local coordinates by
$$
\Delta = - \sum_{i,k=1}^d \frac{1}{\sqrt{|g|}} \frac{\partial}{\partial x_i} \left( \sqrt{|g|} 
g^{i k} \frac{\partial}{\partial x_k} \right),
$$
where $|g|$ is the determinant of the metric $g$, $g^{i k}$ are entries of the 
inverse metric.  As usual we define the space of square integrable functions, $L^2(M)$, 
as the completion of the space of smooth functions, 
$C^{\infty}(M)$, with respect to the norm induced by the inner product
$$
\langle f , g \rangle = \int_M f(x) \, \overline{g(x)} \, \id \mu(x).
$$
Here $\id \mu$ denotes the Riemannian volume element of $M$, which is given in local 
coordinates by $\sqrt{|g|} \id x_1 \id x_2 \ldots \id x_d$. The Laplace operator is 
self-adjoint, non-negative and has compact resolvent.  Therefore there is an orthonormal basis 
$\{\varphi_i \in L^2(M): i \in \N_0\}$ consisting of eigenfunctions
$$
\Delta \varphi_i = \lambda^2_i \varphi_i, \quad \varphi_i \in C^\infty(M)
$$
which is ordered such that $0 = \lambda_0\leq \lambda_1 \leq \ldots$.  Let $e_\lambda(x,y)$ 
be defined as the finite sum $\sum_{\lambda_i < \lambda}\varphi_i(x)\overline{\varphi_i(y)}$. 
The restriction of $e_\lambda$ to the diagonal is called the local
counting function and we will denote it by $N_x(\lambda)=e_\lambda(x,x)$. Integration of $N_x(\lambda)$ over $M$ 
gives the counting function of the Laplace operator
$$
 N(\lambda) = \# \{ i \; | \; \lambda_i < \lambda \}.
$$
 The local Weyl law states that
$$
 N_x(\lambda) = \frac{\omega_d}{(2 \pi)^d} \lambda^{d} + O(\lambda^{d-1}),
$$
where $\omega_d$ is the volume of the unit ball in $\mathbb{R}^d$. This result was 
proved by Levitan in the case of closed Riemannian manifolds. His work was based on 
the study of the cosine transform of the spectral function of the Laplace operator. 
In 1968 H\"ormander~\cite{Hormander2} generalised 
it to the case of pseudo-differential operators of order $m$. The term $O(\lambda^{d-1})$
can not be improved in general as the example of $S^d$ shows. 
Obtaining better estimates of this error term under additional geometric
assumptions is still an active area of research. Various improvements are 
known in the case of compact manifolds with negative sectional 
curvature or under assumptions on the nature of the dynamics of the geodesic flow.
For example, it was shown by  Duistermaat and Guillemin~\cite{Duistermaat} that the assumption 
of the set of periodic trajectories in the cosphere bundle having Liouville measure zero
implies that the $O\left(\lambda^{d-1}\right)$ may be replaced by $o\left(\lambda^{d-1}\right)$.
The most significant for the purposes of this paper is the result of B\'erard~\cite{Berard}, who in 1977
obtained a logarithmic improvement for manifolds with non-positive sectional curvature:
\[N_x(\lambda)= \frac{\mathrm{Vol}(B^*_x)}{(2 \pi)^{d}} \lambda^d +
O\left(\frac{\lambda^{d-1}}{\log \lambda}\right),\] 
where $B_x^*$ is the unit ball in $T^*_x M$, i.e. $\mathrm{Vol}(B^*_x) = \mathrm{Vol}(B^d)$.

It is well known that regularized versions of the counting function have better asymptotic expansions.
An example of a regularized counting function is the $k$-th Riesz means
\[ R_k N_x(\lambda)= k \lambda^{-1} \int_0^{\lambda} (1-\tau \lambda^{-1})^{k-1} 
N_x(\tau) \, \id \tau, \quad k=1,2,\ldots.\]
It shown by H\"ormander in \cite{Hormander-Riesz} and \cite{Hormander2} that the $k$-th Riesz means 
admits an asymptotic expansion with an error term of order  $O\left(\lambda^{d-k-1}\right)$.
Safarov showed in~\cite{Safarov2} that the assumption 
that the set of periodic trajectories in the cosphere bundle under the geodesic flow has Liouville measure zero
implies that this error term may be replaced by $o\left(\lambda^{d-k-1}\right)$.

Another well known asymptotic expansion is that of the mollified counting function (see e.g.
Duistermaat and Guillemin~\cite{Duistermaat}). Namely, let $\rho \in S(\R) $, $\hat{\rho} \in C_0^\infty(\R)$, $\hat \rho(\xi)=1$ 
for all $\xi$ in a neighbourhood of zero, then
\begin{equation}\label{expansion}
\rho \ast N_x(\lambda) \sim \sum_{i=0}^{\infty} a_i(x) \lambda^{d-i}.
\end{equation}
whenever the support of $\hat{\rho}$ is sufficiently small. Here the coefficients are 
local densities and are related directly to the local heat kernel coefficients of the Laplace operator.
Weyl's asymptotic formula is equivalent to the fact that
$$a_0(x)=  \frac{\mathrm{Vol}(B^d)}{(2 \pi)^{d}}.$$ 

It is known (see for example~\cite{Safarov2}) that the existence of such a full asymptotic expansion of the mollified counting
function (\ref{expansion}), independent of the geometric context, is enough to conclude that
for $k<d$ we have
\begin{gather} \label{saf}
 R_k N_x(\lambda) = \sum_{i=0}^{k} \frac{k! (d-i)!}{(d-i+k)!} a_i(x) \lambda^{d-i} + 
O\left(\lambda^{d-k-1}\right), \quad \textrm{as } \lambda \to +\infty.
\end{gather}
(see also \cite{Fulling-Riesz} for the relation of these coefficients
to the heat kernel coefficients). 
For surfaces ($d=2$) of constant negative curvature the Selberg trace formula can be used to 
prove a logarithmic improvement of this formula (see \cite{Hejhal}):
\begin{gather} \label{hej}
R_1 N(\lambda) = \frac{\mathrm{Vol}(M)}{12 \pi} \lambda^2  +a_2+O((\log \lambda)^{-2}).
\end{gather}

A direct combination of B\'erard's asymptotic formula, the expansion (\ref{hej}), and the
method described in~\cite{Safarov2} yields in case $k<d$:
\[ 
R_k N_x(\lambda) = \sum_{i=0}^{k+1} \frac{k! (d-i)!}{(d-i+k)!} a_i(x) \lambda^{d-i} + 
O\left(\frac{\lambda^{d-k-1}}{\log \lambda}\right), \quad \textrm{as } \lambda 
\to+\infty.
\]
In the case $k=1$ and $d=2$ this does however not reduce to Hejhal's estimate (\ref{hej}) but is weaker by a factor
of $\log \lambda$.
The purpose of this paper is to improve the above error estimate to cover (\ref{hej}) and thus to generalize
this estimate to the case of possibly non-constant curvature, higher dimension, and higher Riesz means.
Our main result is the following.

\begin{thm}\label{glowny}
Let $(M,g)$ be a compact $d$-dimensional smooth Riemannian manifold. Assume that 
either $d=2$ and $M$ has no conjugate points or that $M$ has non-positive sectional 
curvature. Then for $k \in \{0,1,2, \ldots, d-1 \}$
\[ R_k N_x(\lambda) = \sum_{i=0}^{k+1} \frac{k! (d-i)!}{(d-i+k)!} a_i(x) \lambda^{d-i} + 
O\left(\frac{\lambda^{d-k-1}}{(\log \lambda)^{k+1}}\right), \quad \textrm{as } 
\lambda \to +\infty,\]
where the densities $a_i(x)$ are the same as in (\ref{expansion}) and may be calculated explicitly. For $k \geq d$ one has 
\[ R_k N_x(\lambda) = \sum_{i=0}^{d} \frac{k! (d-i)!}{(d-i+k)!} a_i(x) \lambda^{d-i} + 
O\left( \lambda^{-1+ \varepsilon}\right), \quad \textrm{as } 
\lambda \to +\infty,\]
for $\varepsilon >0$.
\end{thm}

Of course integration with respect to $x$ over $M$ yields corresponding estimates for the counting function.

Apart from the theoretical significance estimates of the form given in Theorem~\ref{glowny} have practical applications in numerical computations of eigenvalues.
Some algorithms, such as the method of particular solutions (\cite{Fox:1967,Betcke:2005}, or  \cite{Strohmaier} on manifolds) produce a list of eigenvalues of the Laplace operator and one  would then like to have a method to check whether or not eigenvalues are missing in this list. 
Whereas the error estimate in Weyl's law is too large to 
detect a missing eigenvalue, the error estimates in some of the higher Riesz means will be sensitive to a change
in the counting function by a positive integer. Averaged versions of the Weyl law are being used in numerical computations of large sets of eigenvalues, for example of Maass eigenvalues (see e.g. \cite{Jorgenson-Smajlovic-Then} where such Weyl laws are derived from the Selberg trace formula for certain congruence groups and subsequently being used in this context).
This method is sometimes referred to as Turing's method.\\
\noindent
{\bf Acknowledgements}\\
\thanks{We would like to thank Andreas Str\"ombergsson and Andrew Booker for pointing out reference \cite{Hejhal} to us. 
We are also grateful to Yuri Safarov for comments on earlier versions of this paper and for pointing out simplifications of the argument.}

\section{Proof and main estimates}
 
Throughout the text $\rho \in \mathcal{S}(\R)$ will be be a real valued Schwartz function 
such that \vspace{0.2cm}\\
 \indent { (i):} $\mathrm{supp}(\hat{\rho}) \subset [-1,1]$, \vspace{0.2cm}\\
 \indent { (ii):} $\hat \rho(\xi) = 1$ for all $|\xi|<1/2$, and\vspace{0.2cm}\\
 \indent {  (iii):} $\rho$ is even.\vspace{0.2cm}\\
Here $\hat{\rho}$ denotes the Fourier transform of $\rho$ defined by
$$
\hat{\rho}(\xi) = \int_\R \rho(t) e^{-i t \xi} \, \id t.
$$
For $T>0$ we denote by $\rho_T$ the rescaled function $\rho_T(t) =T \rho(T t)$, so that
$\hat{\rho}_T(t)=  \hat{\rho}(t/T)$. We therefore have 
$\mathrm{supp}(\hat{\rho}_T) \in [-T,T]$. \\
Then
\begin{equation} \label{splitting}
N_x(\lambda) = N_x \ast \rho_{\epsilon}(\lambda)+[N_x \ast (\rho_{T} - \rho_{\epsilon})](\lambda)+ 
[N_x \ast (\delta -\rho_T)](\lambda),
\end{equation}
where $\epsilon>0 $ is a small parameter, which is smaller than the injectivity radius at $x$. Under the stated assumptions the 
first term has a full asymptotic 
expansion i.e
\begin{equation} \label{fullexpansion}
N_x\ast \rho_\epsilon(\mu) \sim \sum_{i=0}^{\infty} a_i(x) \mu^{d-i}
\end{equation}
for $\mu \to \infty$ as it was proved
in~\cite{Duistermaat} and~\cite{Ivrii}. 
Because $N_x$ is supported on the positive semi-axes and $\rho_\epsilon$ is the Schwartz function $N_x\ast \rho_\epsilon(\mu)$ is rapidly decreasing as $\mu \to -\infty$. 

The $k$-th Riesz means of the local counting function is then given by
\begin{eqnarray} \label{equ6}
R_k N_x(\lambda)&=&  \int_{-\infty}^{\lambda} (1-\tau \lambda^{-1})^{k} \, \id N_x(\tau)= \lambda^{-k} \int_{-\infty}^{\lambda} (\lambda-\tau )^{k} \, \id N(\tau) \\ \nonumber
&=& \lambda^{-k} k! \ \chi_+^k \ast N_x'(\lambda)= \lambda^{-k} k! \ \chi_+^{k-1} \ast N_x(\lambda).
\end{eqnarray}
Here $\chi_{+}^{\alpha}(r)$ is the analytic continuation of $r^{\alpha}_{+}/\Gamma(\alpha+1)$
in the parameter $\alpha$,  as described for example in~\cite{Hormander}. Let us apply the Riesz means operator to~(\ref{splitting}):
\begin{eqnarray} \label{podzialN}
(R_k N_x)(\lambda) &= & \lambda^{-k} k! \biggl [ \left(\chi_+^{k-1} \ast N_x \ast \rho_{\epsilon}\right)(\lambda) + \left(\chi_+^{k-1} \ast 
N_x \ast \left(\rho_{T} - \rho_{\epsilon} \right) \right)(\lambda)  \\ \nonumber
& + &   \left( \chi_+^{k-1} \ast N_x \ast \left( \delta -\rho_T \right) \right) (\lambda) \biggr ].
\end{eqnarray}
Convolution with the distribution $\chi_+^k$ may be understood as a repeated integral from $-\infty$ to $\lambda$, therefore
$$
\left[\chi_+^{k-1} \ast \left(N_x \ast \rho_{\epsilon}\right)\right](\lambda)= \int_{-\infty}^\lambda \int_{-\infty}^{\lambda_1} \ldots \int_{-\infty}^{\lambda_{k-1}} (N_x \ast \rho_{\epsilon} )(\lambda_k) \, \id \lambda_k \ldots \id \lambda_2 \id \lambda_1, \quad \textrm{for } k\geq 1.
$$
Together with the full asymptotic expansion~(\ref{fullexpansion}) we obtain
\begin{equation}\label{nearorigin}
\lambda^{-k} k! \left[\chi_+^{k-1} \ast \left(N_x \ast \rho_{\epsilon}\right)\right](\lambda)=  \sum_{i=0}^{d} \frac{(d-i)!a_i(x)}{(d-i+k)!} \lambda^{d-i} + O(\lambda^{-1+\varepsilon})
\end{equation}
for any $\varepsilon>0$ as $\lambda \to \infty$. 

Our main result is derived from the following estimates of the various terms appearing in (\ref{splitting})
and (\ref{podzialN}). For convenience we use the following notation
$$
 \cla = (1+|\lambda|).
$$

\begin{prop} \label{diffest}
 For fixed $\epsilon>0$ and $k\geq 0$ there exists a constant $c>0$ such that for all $T \geq 1$ we have
 \begin{gather*}
 | (N_x \ast \chi^{k-1}_+ \ast (\rho_T-\rho_\epsilon) )(\lambda)| \leq  e^{c T} \cla^{\frac{d-1}{2}}.
\end{gather*}
\end{prop}

\begin{prop} \label{twotermdiffest}
For fixed $\epsilon>0$ there exists $c>0$ such that for all $T \geq 1$ we have
\begin{gather*}
 | N_x \ast (\delta - \rho_T)(\lambda)| \leq  \frac{c}{T} \cla^{d-1} + e^{c T} \cla^{\frac{d-1}{2}}.
\end{gather*}
\end{prop}

We postpone the proof of these two propositions for the moment
and show how they imply the result.
The splitting (\ref{splitting}), the expansion (\ref{fullexpansion}), Proposition \ref{diffest} in the case $k=0$ together with
Proposition \ref{twotermdiffest} immediately imply that there exists a $c>0$ such that for $T \geq 1$:
\begin{equation} \label{diffest2}
|N_x(\lambda) - H(\lambda) \sum_{i=0}^{d-1} a_i(x) \lambda^{d-i} |  \leq \frac{c}{T} \cla^{d-1} +  e^{c T}  \cla^{\frac{d-1}{2}}.
\end{equation}
Here $H(\lambda)$ is the Heaviside step function defined by $H(\lambda)=1$ for $\lambda \geq 0$ and $H(\lambda)=0$ for $\lambda<0$.
Now we use a proposition derived by Safarov in \cite{Safarov2}, which we slightly adapt to our situation.
\begin{prop} \label{propsafarov}
 Suppose that $\nu_1,\nu_2 \in \R$ and suppose $a_0,a_1,\ldots,a_d \in \mathbb{R}$. Then
there exists a constant $C>0$ depending only on $\rho, \nu_1$, $\nu_2$, and $(a_i)_{i=1,\ldots,d}$ such that the following statement holds.
Suppose $N$ is a function of locally bounded variation  that is supported in $[0,\infty)$, and assume that:
$$
 | N(\lambda)- H(\lambda)\sum_{i=0}^{d} a_i \lambda^{d-i} | \leq C_1 \cla^{\nu_1} + C_2 \cla^{\nu_2}.
$$ 
Then, for all $T\geq 1$ and $\lambda \geq 0$:
$$
\left| \int_{-\infty}^\lambda N(\mu) - N \ast \rho_T (\mu) \, \id \mu \right| \leq \frac{C}{T} (C_1 \cla^{\nu_1} + C_2 \cla^{\nu_2} + 1).
$$
\end{prop}
\begin{proof}
 Our assumptions on $\rho$ imply that $\int \rho(t) \id t =1$ and $\int \rho(t) t^k \id t=0$ for all $k \in \N$. We have,
 \begin{gather*}
  \int_{-\infty}^\lambda \left( N(\mu) - N \ast \rho_T (\mu) \right) \id \mu \\ = \int \rho(t) 
  \int_{\lambda-T^{-1} t}^\lambda N(\tau) \id \tau \id t  =
    \int \rho(t) \int_{\lambda-T^{-1} t}^\lambda (N(\tau)-\sum_{i=0}^{d} a_i \tau^{d-i}) \id \tau \id t \\=
  \int \rho(t) \int_{\lambda-T^{-1} t}^\lambda (N(\tau)-H(\tau)\sum_{i=0}^{d} a_i \tau^{d-i}) \id \tau \id t \\ -\int \rho(t) \int_{\lambda-T^{-1} t}^\lambda H(-\tau)(\sum_{i=0}^{d} a_i \tau^{d-i}) \id \tau \id t.
 \end{gather*}
Since $\rho$ is rapidly decreasing the modulus of the last term is bounded by $\frac{C_3}{T}$ for some $C_3>0$ depending only on $\rho$ and the $a_i$.
Therefore,
$$
 \left| \int_{-\infty}^\lambda \left( N(\mu) - N \ast \rho_T (\mu) \right) \id \mu \right| \leq   
 \int \left | \rho(t) \int_{\lambda-T^{-1} t}^\lambda \left(C_1 \langle \tau \rangle^{\nu_1} + C_2 \langle \tau \rangle^{\nu_2} \right) \id \tau \right| \id t + \frac{C_3}{T}.
$$
Using the triangle inequality and the fact that
$\langle \tau + \lambda \rangle^\nu \leq \langle \tau \rangle^{|\nu|} \langle \lambda \rangle^\nu$ one obtains
\begin{gather*}
 \int \left| \rho(t) \int_{\lambda-T^{-1} t}^\lambda \langle \tau \rangle^{\nu} \id \tau \right| \id t \leq
 \int \left|\rho(t) \int_{0}^{T^{-1} t} \langle \tau - \lambda \rangle^{\nu} \id \tau \right| \id t \leq 
  T^{-1} \left(\int \left|\rho(t) t \right| \langle t \rangle^{|\nu|} \id t \right) \cla^\nu.  
\end{gather*}
This shows the proposition with for all $\lambda\geq 0$ with constant
$$
 C = \int \left| \rho(t) t \right| \langle t \rangle^{\mathrm{max}\{\nu_1,\nu_2\}} \id t + C_3.
$$
\end{proof}

Repeated application of Proposition~\ref{propsafarov}  to the estimate (\ref{diffest2}), using Proposition~\ref{diffest} and (\ref{fullexpansion}), shows that for any integer $k\geq 0$
there exists a $c>0$ such that for $T\geq 1$:
$$
 \left| \left[\chi_+^{k-1} \ast \left(N_x \ast \left(\delta -\rho_T\right)\right)\right](\lambda) \right| \leq 
 \frac{c}{T^{k+1}} \cla^{d-1} + e^{c T} \cla^{\frac{d-1}{2}}.
$$
When we substitute this estimate into (\ref{podzialN}), use the expansion (\ref{nearorigin}), and Proposition \ref{diffest} 
we get that 
$$
\left|R_k N_x(\lambda) - \sum_{i=0}^{d-1} \frac{k! (d-i)!}{(d-i+k)!} a_i \lambda^{d-i} \right| \leq  
\frac{c}{T^{k+1}} \cla^{d-1-k} +  e^{c T}  \cla^{\frac{d-1}{2}}
$$
This estimate is valid for $T \geq 1$. If we take $T= \alpha \log \lambda $ for some small $\alpha >0 $ 
of obtain Theorem~\ref{glowny} for large $\lambda$.\\

The Laplace operator is non-negative and thus the local counting function is 
supported on the positive half line $\lambda \geq 0$. Let us define
$$
N^{odd}_x(\lambda):= N_x(\lambda)-N_x(-\lambda), \quad N^{neg}_x(\lambda):=N_x(-\lambda).
$$
It is clear that the functions just defined sum up to $N_x$. The convolution, 
$[(N_x^{odd}) \ast (\rho_T-\rho_\epsilon)](\lambda)$, admits the same asymptotics 
as $[(N_x) \ast (\rho_T-\rho_\epsilon)](\lambda)$ as $\lambda \to \infty$.
Moreover,
$[(N_x^{odd}) \ast (\rho_T-\rho_\epsilon)](\lambda)$ is the restriction to the diagonal of the integral kernel of the operator
\begin{equation} \label{evenoperator}
(2\pi)^{-1}   \int_{\R} (\hat{\rho}_T(t)- \hat{\rho}_\epsilon(t)) 
t^{-1} \sin({t \lambda }) \cos( t \sqrt{\Delta}) \, \id t.
\end{equation}
We have
\begin{equation} \label{orderchange}
\chi_+^{k-1} \ast \left( N_x \ast (\rho_T-\rho_\epsilon) \right) (\lambda) =  \left( N_x \ast \left( \chi_+^{k-1} \ast (\rho_T-\rho_\epsilon) \right)\right) (\lambda).
\end{equation} 

Let as usual $p_{\alpha,\beta}$ be the Schwartz space semi-norms defined by
$p_{\alpha,\beta}(f) = \sup_{x} |x^\alpha \partial_x^\beta f|$. Then it easy to check that
for all $T>1$ we have
$$p_{\alpha,\beta}(\hat \rho_T-\hat \rho_\epsilon) \leq C_{\alpha,\beta} T^{\alpha}.$$

Thus, since $N_x$ is supported in the half line,
$N_x \ast (\chi^{k-1}_+ \ast (\rho_T-\rho_\epsilon) )(\lambda) = O(e^{c T} \lambda^{-\infty})$ as $\lambda \to -\infty$
for any $c > 0$.
Hence, $N_x \ast (\chi^{k-1}_+ \ast (\rho_T-\rho_\epsilon) )(\lambda)$ and
$N_x^{odd} \ast (\chi^{k-1}_+ \ast (\rho_T-\rho_\epsilon) )(\lambda)$ differ by  a function of order $O(e^{c T} \lambda^{-\infty})$ as 
$\lambda \to \infty$ for any $c>0$. We use the identity 
\[
\widehat{\chi_{+}^{k-1}}(\xi)=  \frac{(-i)^k}{\sqrt{2 \pi}}( \xi - i 0 )^{-k},
\]
to express the function
$N_x^{odd} \ast (\chi^{k-1}_+ \ast (\rho_T-\rho_\epsilon) )(\lambda)$ as the restriction to the diagonal 
of the integral kernel of
$$
(2\pi)^{-1} \mathrm{Re}  \int_{\R} \frac{\hat{\rho}_T(t)- \hat{\rho}_\epsilon(t)}{ 
(it)^{k+1}} e^{it \lambda } \cos( t \sqrt{\Delta}) \, \id t.
$$

To estimate this integral we will as usual be exploiting the properties of a suitable parametrix
for the operator $e^{-i t \sqrt{\Delta}}$. In our case we will use the parametrix for $\cos(t \sqrt{\Delta})$ that was constructed by B\'erard in \cite{Berard}
and which we describe in the following.
Let  $\pi: \tilde{M} \to M$ be the universal 
cover of $M$ and let $\Gamma$ be its group of automorphisms, so that $M \cong \tilde M / \Gamma$.
Denote by $\tilde{C}(t,x,y)$ the integral kernel of 
$\cos(t \sqrt{\tilde\Delta})$, where $\tilde \Delta$ is the Laplace operator on the non-compact complete manifold $\tilde{M}$, then
$$
\tilde{C}(t,x,y)= C_0 \sum_{l=0}^N (-1)^l 4^{-l} u_l (x,y) |t|  
\left. \frac{\left(d(x,y)^2-t^2\right)^{l - \alpha}_-}{\Gamma({l +1 - \alpha})} \right|_{\alpha= \frac{d+1}{2}} + \tilde{\epsilon}_{N}( t,x,y),
$$
where $d(x,y)$ denotes the distance between $x$ and $y$ on $\tilde{M}$. The regularizations of the distributions $x^{\alpha}/\Gamma(\frac{\alpha +1}{2}) $, $x^{\alpha}/\Gamma(\alpha) $, $x^{-\alpha}$ are described for example in \cite{Hormander}.
Moreover the functions $u_l$ and $\Delta_y^m u_l$ all have at most exponential 
growth as $d(x,y)$ tends to infinity, i.e. for all $l$, $m$ exists a $c>0$ such that
\begin{equation} \label{estimate}
|\Delta^m_y u_l(x,y)| \leq c\, e^{c \,d(x,y)},
\end{equation}
Moreover, for $N\geq \lfloor d/2 \rfloor +3$, the error term is continuous and 
bounded uniformly in $x$ and $y$ as follows
\begin{equation} \label{error}
| \tilde{\epsilon}_N(t,x,y)| \leq c_N \, e^{c_N | t |}.
\end{equation}
The distributional kernel of $\cos(t \sqrt{\Delta})$ on $M$ is given by:
\begin{equation}\label{wavekernel}
C(t,x,y)= \sum_{\gamma \in \Gamma} \tilde{C}(t,x, \gamma y).
\end{equation}

The integral kernel of $\cos(t \sqrt{\tilde \Delta} )$ has the finite propagation speed property, 
i.e. it is supported where $d(x,y) \leq t$. This property implies that for fixed 
$x$ and $y$ the number of terms in~(\ref{wavekernel}) is finite for every $t$. In fact, the assumptions
on the curvature for $d>2$ and on the absence of conjugate points for $d=2$ imply that
the number of terms in~(\ref{wavekernel}) 
is $O(e^{c |t|})$ for some $c>0$ which depends on the geometry of the manifold. 
The integral kernels of $\cos(t \sqrt{\Delta})$, $(\rho_T-\rho_{\epsilon})(t)$ and $\mathrm{Re}((it)^{-k-1} e^{it \lambda})$ 
are even functions with respect to the independent variable $t$. Therefore we may restrict our integration to the positive 
semi-axes and in order to get a bound on~(\ref{evenoperator}) we need to estimate terms of the following form:
\begin{equation} \label{term}
 \mathrm{Re} \int_{0}^{\infty} \frac{\hat{\rho}_T (t) -\hat{\rho}_\epsilon (t)}{(it)^k} \left. \frac{
  \left(d(x,\gamma x)^2-t^2\right)_{-}^{l - \alpha}}{\Gamma(l +1- \alpha)} \right|_{\alpha= \frac{d+1}{2}}  e^{it\lambda} \, \id t
\end{equation}
Note that for $\gamma = \mathrm{id}$ and odd dimension $d=2m-1$ the distribution 
$\left. t^{l+\alpha}/\Gamma(l+\alpha+1) \right|_{\alpha=-m}$ is supported at the origin for $l<m$, 
therefore in this case the pairing above is 0, since $\hat{\rho}_T-\hat{\rho}_\epsilon$ is supported 
for $|t|>\epsilon/2$ otherwise for $\gamma = \mathrm{id}$ we have the Fourier transform of a 
$C_0^{\infty}(-T,T)$ function. 
In the following we use the notation $O_T(g(x))$ for $O(g(x))$ in case the implied constant can be chosen
independent of $T$, i.e. $f(x,T)=O_T(g(x))$ if $| f(x,T) | \leq C |g(x)|$ with $C$ not dependent on $T$.
The estimates will be based on the following two Lemata. Assume that $\eta$ is an even real valued Schwartz function such that 
$\hat \eta \in C^\infty_0([-1,1])$
is an even Schwartz function and let $\hat \eta_T$ be the rescaled function $\hat \eta_T(t)=\hat \eta(t/T)$.

\begin{lem} \label{lemest11}
 Let $\epsilon>0$ and $k \geq 0$. If $k>0$ we suppose furthermore that $\hat \eta-1$ vanishes of order $k$ at zero.
 Then,
 $$\int_0^\infty \left( \hat \eta_T(t) - \hat \eta_\epsilon(t) \right) t^{-k} e^{i t \lambda} \id t = O_T(\lambda^{-\infty})$$
 as $\lambda \to \infty$ for $T \geq 1$.
\end{lem}

\begin{proof}
 One checks directly that for any integer $m>0$ the $L^1$-norm of 
 $\frac{\id^m}{\id t^m}(\hat \eta_T(t) - \hat \eta_\epsilon(t)) t^{-k})$
 is bounded in $T$ for $T \geq 1$. The Lemma then follows by integration by parts.
\end{proof}

\begin{lem} \label{lemest12}
  For any fixed $\epsilon>0$, $k\geq 0$, and $m \in \mathbb{R}$ there exists an $L>0$ such that
  $$T^{-L}\int_0^\infty \hat \eta_T(t)\; t^{-k}\; \frac{(R^2-t^2)_-^{m}}{\Gamma(m +1)} \; e^{i \lambda t} \; dt= 
  O_{T,R}\left( 1+|\lambda|^{-m-1} \right)$$
  as $|\lambda| \to \infty$ for all $T \geq 1$ and $R$ with $T \geq R > \epsilon$.
\end{lem}
\begin{proof}
 For $m \geq 0$ the estimate follows for $L = 2m + 1 -k$ immediately from the support properties of $\hat \eta_T$
 and the fact that $\hat \eta_T$ is bounded. It remains to show the estimate in case $m<0$. Therefore, assume $m<0$. 
 Let $\phi_T \in C^\infty_0(\mathbb{R}_+)$ such that $\phi_T=0$ in a neighborhood of zero, $\phi_T=1$
 on $[R,T]$ and $\phi_T=0$ on $[2T,\infty)$, such that $\phi_T^{(\beta)} \leq C_\beta$ uniformly in $T$ and $d$
 for $T \geq 1$ and $T \geq R > \epsilon$.
 Then one shows easily that
 the Schwartz semi-norms of $\phi_T \hat \eta_T t^{-k} (t + R)^m$ satisfy
 $$
  p_{\alpha,\beta}\left( \phi_T \hat \eta_T t^{-k} (t + R)^m \right) \leq  C_{\alpha,\beta} T^{\alpha}
 $$
 for $T \geq 1$ and $T \geq R > \epsilon$, where $C_{\alpha,\beta}$ is independent of $T$ and $R$.
 If $\psi_T$ is the Fourier transform of the function $\phi_T \hat \eta_T t^{-k} (t + R)^m$ we therefore have
 $$
  p_{\alpha,\beta} \left( \psi_T \right) \leq \tilde C_{\alpha,\beta} T^{\beta+2}
 $$
 On the other hand the Fourier transform of
 $$
  \hat \eta_T t^{-k} \frac{(R^2-t^2)_-^{m}}{\Gamma(m +1)}
 $$
 is the convolution of $\psi_T$ with the Fourier transform of $\frac{(R-t)_-^{m}}{\Gamma(m + 1)}$.
 The Fourier transform of the distribution $\frac{(R-t)_-^{m}}{\Gamma(m + 1)}$
 can be computed explicitly and is a locally integrable function of order $O_R(\lambda^{-m-1})$ as $\lambda \to \infty$
 To estimate this convolution we use the well known inequality
 \begin{equation} \label{convinequ}
  (1+|\mu-\lambda|)^{-m-1} \leq (1+|\mu|)^{|m+1|}(1+|\lambda|)^{-m-1}, 
 \end{equation}
 and the fact that
 $$
  \int \psi_T(\mu) (1+ |\mu|)^{|m+1|} \id \mu
 $$
 can be bounded by a multiple of $\sup_\mu \langle \mu \rangle^q | \psi_T(\mu) |$ for all $q > (|m+1|+1)$. It follows that 
 the convolution is of order $O_{T,R}(T^2 \lambda^{-m-1})$ and we may therefore choose $L=2$.
\end{proof}

{\bf Proof of Proposition \ref{diffest}:}\\ 
For each $T\geq1$ the number of non-zero terms in the sum 
(\ref{wavekernel}) in finite, the number of terms grows at most exponentially fast with $T$. Moreover, 
the estimate~(\ref{estimate}) implies that there is a constant $c_l$ independent on $x$ 
such that
\begin{equation} \label{potrzebne1}
|u_l(x,\gamma x)| \leq c_l \exp( c_l T)
\end{equation}
on the support of $\widehat{\rho_T}$.
The above Lemmata, applied with $\eta(\xi) = \hat \rho(\xi)$, together with these growth estimates show that for some $c>0$
\begin{equation} \label{awayorigin}
N_x^{odd} \ast (\chi^{k-1}_+ \ast (\rho_T-\rho_\epsilon) )(\lambda) \leq e^{c T} \cla^{\frac{d-1}{2}},
\end{equation}
for $T\geq 1$,
This implies the statement of the Proposition.\qed\\

{\bf Proof of Proposition \ref{twotermdiffest}:}\\ 
We would now like to estimate 
$N_x \ast (\delta - \rho_T)$. This can be done as follows using a Fourier Tauberian theorem.
Let $\tilde \rho \in \mathcal{S}(\R)$ be a non-negative even Schwartz function such that the Fourier
transform $\hat{\tilde \rho}$ is supported in the interval $[-1/2,1/2]$ and such that
$\hat{\tilde  \rho}(0) =1$.
Let $\tilde \rho_T$ be the rescaled function defined by $\tilde \rho(t) = T  \tilde \rho(T t)$.
Following Safarov (\cite{Safarov1}) we define $\tilde \rho_{1,0} \in \mathcal{S}(\R)$
by
$$
 \tilde \rho_{1,0}(t) = \int_t^\infty  \tau \tilde \rho(\tau) \id \tau,
$$
so that
$$
 \hat{\tilde \rho}_{1,0}(\xi) = -\frac{1}{\xi} \frac{d}{d \xi} \hat{\tilde \rho}(\xi).
$$
Define $\tilde \rho_{T,0}$ by $\tilde \rho_{T,0}(t)=T \tilde \rho_{1,0}(T t)$.
Lemma \ref{lemest11} and Lemma \ref{lemest12} applied with $\eta(\xi)=\hat{\tilde \rho}_{1,0}$ and $k=0$, 
together with the above growth estimates then imply the bound
\begin{equation} \label{tauberbound}
 N_x' \ast {\tilde \rho}_{T,0}(\lambda) \leq e^{{c_1} T} \cla^{\frac{d-1}{2}} + c_1 \cla^{d-1}
\end{equation}
for all $T\geq 1$. The term  $c_1 \cla^{d-1}$
appears here because of the contribution of the identity element in $\Gamma$.

Under these conditions the Fourier Tauberian Theorem 1.3 in \cite{Safarov1} states that
$$
 | N_x \ast (\delta - \tilde \rho_T)(\lambda)| \leq \frac{C}{T} N_x' \ast {\tilde \rho}_{T,0}(\lambda),
$$
for all $T \geq 1$ with a constant $C>0$ that does not depend on $T$.
Therefore, there exists $c>0$ such that for all $T \geq 1$:
\begin{gather*}
 | N_x \ast (\delta - \tilde \rho_T)(\lambda)| \leq  \frac{c}{T}\cla^{d-1} + e^{c T} \cla^{\frac{d-1}{2}}  .
\end{gather*}

The support properties of the Fourier transforms of $\rho_T$ and $\tilde \rho_T$ imply that
$\rho_T \ast \tilde \rho_T = \tilde \rho_T$.
Hence, we have
$$
| N_x \ast (\delta - \rho_T)(\lambda)|  \leq |N_x \ast (\delta - \tilde \rho_T)(\lambda)| + 
|(N_x \ast (\delta - \tilde \rho_T) \ast \rho_T) (\lambda)|.
$$
Using $\langle \lambda-\mu \rangle^\alpha \leq \langle \lambda\rangle^{|\alpha|} \langle \mu \rangle^\alpha$ one derives the bound
$$\int | \rho_T(\mu) | \langle \lambda-\mu \rangle^{\alpha} d\mu = O_T(\cla^\alpha).$$
Therefore,
\begin{gather*}
| N_x \ast (\delta - \tilde \rho_T) \ast \rho_T(\lambda) | \leq \frac{C}{T} \left( c_3 \cla^{d-1} + 
c_2 e^{c_1 T} \cla^{\frac{d-1}{2}} \right),
\end{gather*}
Summarizing, there exists a $c>0$ such that for all $T \geq 1$ we have 
\begin{equation} \label{diffest1}
[N_x \ast(\delta - \rho_T)] \leq \frac{c}{T} \cla^{d-1}+ e^{c T} \cla^{\frac{d-1}{2}}.
\end{equation}
This implies the proposition.\qed \\

\appendix


\begin{thebibliography}{9}
\bibitem{Berard} P.~H.~B\'erard: \emph{On the Wave Equation on a Compact Riemannian 
Manifold without Conjugate Points}, Mathematische Zeitschrift, Springer-Verlag, 1977.

\bibitem{Betcke:2005}
T. Betcke and L.~N. Trefethen, \emph{Reviving the method of particular
  solutions.}, SIAM Review \textbf{47} (2005), no.~3, p469 -- 491.

\bibitem{Duistermaat} J.~Duistermaat, V.~Guillemin: \emph{The spectrum of positive 
elliptic operators and periodic bicharacteristics.} Inventiones math. 29, 39-79, 1975.

\bibitem{Fox:1967}
L.~Fox, P.~Henrici, and C.~B. Moler, \emph{Approximations and bounds for
  eigenvalues of elliptic operators}, SIAM J. Numer. Anal. \textbf{4} (1967),
  89--102.

\bibitem{Fulling-Riesz} S.~A.~Fulling {\sl (with an appendix by R.A.~Gustafson)}: 
\emph{Some properties of Riesz means and spectral expansions}, 
Electronic Journal of Differential Equations, No. 06 (1999), pp. 139.

\bibitem{Hejhal} 
D.~A.~Hejhal: \emph{The Selberg trace formula for $PSL(2,R)$}. Vol. I, Lecture Notes in Mathematics, Vol. 548, 
Berlin, New York: Springer-Verlag, 1976.

\bibitem{Hormander} L.~H\"{o}rmander: \emph{Linear Partial Differential Equations}, 
Springer-Verlag, Berlin, 1963.

\bibitem{Hormander-Riesz} L.~H\"{o}rmander: \emph{On the Riesz means of spectral function and 
eigenfunction expansion for elliptic differential operators}, Lecture at the Belfer Graduate School, Yeshiva University, 1966 (preprint).

\bibitem{Hormander2} L.~H\"{o}rmander: \emph{The spectral function of an elliptic 
operator}, Acta Math. 121,193-218, 1968.

\bibitem{Ivrii} V.~Ivrii:, \emph{Precise spectral asymptotics for elliptic operatos}, 
Lect. Notes Math., No. 1100, Springer, Berlin, 1984.

\bibitem{Jorgenson-Smajlovic-Then} J.~Jorgenson, L.~Smajlovic, H. Then, \emph{On the distribution of eigenvalues of Maass forms on certain moonshine groups}
to appear in Math. Comp.

 
\bibitem{Safarov1} Y.~Safarov: \emph{Fourier Tauberian theorems and applications},  Journal of Functional Analysis 185, 2001,

\bibitem{Safarov2} Y.~Safarov: \emph{Riesz means of the distribution function of the 
eigenvalues of an elliptic operator}, Zapiski Nauchnykh Seminarov Leningradskogo 
Otdeleniya Matematicheskogo Instituta im. V.A. Steklova AN SSSR, Vol. 163, pp. 143-145.
1987,

\bibitem{Strohmaier} A.~Strohmaier, V.~Uski: \emph{An algorithm for the computation of eigenvalues, spectral 
zeta functions and zeta-determinants on hyperbolic surfaces},  Comm. Math. Phys. 317, 2013.


\end{thebibliography}
\end{document}